\documentclass[reqno,12pt]{amsart}
\usepackage[margin=1.1in, footskip=1cm]{geometry}
\usepackage{amsmath, amsthm, amssymb, booktabs}
\usepackage{mathrsfs}
\usepackage[usenames,dvipsnames]{color}
\pdfpagewidth=\paperwidth
\pdfpageheight=\paperheight

\newtheorem{theorem}{Theorem}[section]
\newtheorem{lemma}[theorem]{Lemma}

\theoremstyle{definition}

\theoremstyle{remark}

\numberwithin{equation}{section}

\newcommand{\mmod}[1]{\,\,({\rm{mod}}\,\,#1)}

\def\d{{\partial}}

\def\le{\leqslant} \def\ge{\geqslant}

\def\d{{\,{\rm d}}}

\makeatletter
\@namedef{subjclassname@2020}{\textup{2020} Mathematics Subject Classification}
\makeatother

\begin{document}
\title[Sharp mean value estimates]{Sharp mean value estimates in the absence of\\ translation-dilation invariance}
\author[T. D. Wooley]{Trevor D. Wooley}
\address{Dept.~of Mathematics, Purdue University, 150 N. University Street, West 
Lafayette, IN 47907-2067, USA}
\email{twooley@purdue.edu}
\date{}

\begin{abstract} We establish sharp mean value estimates for various exponential sums lacking translation-dilation 
invariance properties. For example, suppose that $k\ge 1$ and $s>0$. Then, for each $\varepsilon>0$, one has
\[
\int_{[0,1)^k}\biggl| \sum_{1\le h\le X}\sum_{1\le z\le X}e(hz\alpha_1+\ldots +hz^k\alpha_k)\biggr|^{2s}
\d\boldsymbol \alpha \ll X^\varepsilon\bigl( X^s+X^{2s-k(k+3)/2}\bigr),
\]
confirming the main conjecture for a new family of exponential sums. Our methods are based on the application of 
very recent work of the author employing the nested efficient congruencing method.
\end{abstract}

\maketitle

\section{Introduction} Transformational progress has been achieved in the last decade or so concerning mean value 
estimates for exponential sums, achieving in some cases the sharp estimates conjectured to hold (see \cite{BDG2016, 
GZK2020, Woo2016, Woo2019}, for example). This progress is almost entirely limited to Vinogradov systems of 
equations, and analogues of multidimensional type also possessing translation-dilation invariance properties. Our 
goal in this paper is to expand the list of families amenable to such sharp analysis beyond the class of 
translation-dilation invariant systems. This goal was achieved for a new class of exponential sums in one variable by 
the author in recent work \cite{Woo2026a}. We turn in this memoir to a family of bivariate exponential sums, also not 
of translation-dilation invariant type, for which we achieve similar success.\par

We begin our account of the new systems amenable to our methods with a consideration of the exponential sum
\begin{equation}\label{1.1}
f_k(\boldsymbol \alpha;X)=\sum_{1\le h\le X}\sum_{1\le z\le X}e(h\alpha_1+hz\alpha_2+\ldots +hz^{k-1}\alpha_k).
\end{equation}
Here, as is conventional, we write $e(\theta)$ for ${\rm e}^{2\pi {\rm i}\theta}$. We shall be interested in the mean 
value
\begin{equation}\label{1.2}
B_{s,k}(X)=\int_{[0,1)^k}|f_k(\boldsymbol \alpha;X)|^s\d\boldsymbol \alpha .
\end{equation}

\begin{theorem}\label{theorem1.1}
Suppose that $k\ge 2$ and $s$ is a positive real number. Then, for each $\varepsilon>0$, one has
\begin{equation}\label{1.3}
B_{s,k}(X)\ll X^\varepsilon \bigl( X^s+X^{2s-k(k+1)/2}\bigr) .
\end{equation}
\end{theorem}

Note that a trivial estimate for the exponential sum $f_k(\boldsymbol \alpha;X)$ yields the bound 
\[
|f_k(\boldsymbol \alpha;X)|\le X^2,
\]
since the sum \eqref{1.1} has at most $X^2$ summands. Thus, the first term in the upper bound \eqref{1.3} reflects 
square-root cancellation on average. Meanwhile, the sum of the degrees of the monomials $hz^{j-1}$ $(1\le j\le k)$ 
occurring in the argument of the exponential sum \eqref{1.1} is equal to $\tfrac{1}{2}k(k+1)$. The set of 
$\boldsymbol \alpha$ with
\[
|\alpha_j| \le \frac{1}{(100k)} X^{-j} \quad (1\le j\le k)
\]
makes a contribution to the integral in (1.2) establishing the lower bound
\[
B_{s,k}(X) \gg X^{2s-k(k+1)/2}.
\]
A consideration of diagonal solutions consequently reveals that the upper bound \eqref{1.3} is sharp up to the factor 
of $X^\varepsilon$.\par

We emphasise that the parameter $s$ in Theorem \ref{theorem1.1} is not constrained to be an even integer. However, 
when $s$ is an even integer, say $s=2t$, then it follows via orthogonality that $B_{s,k}(X)$ counts the number of 
integral solutions of the system of equations
\begin{equation}\label{1.4}
\sum_{i=1}^{t} h_i z_i^{j-1}=\sum_{i=t+1}^{2t}h_i z_i^{j-1}\quad (1\le j\le k).
\end{equation}
While the system of equations \eqref{1.4} is translation-invariant with respect to the shift $z_i\mapsto z_i+a$ 
$(1\le i\le 2t)$, there is no translation invariance property for the variables $h_i$. Thus, the system \eqref{1.4} does not 
belong to the class of translation-dilation invariant systems amenable to methods based on efficient congruencing 
and decoupling. Earlier methods based on the use of polynomial identities and divisor function estimates are 
available. These recover the conclusion of Theorem \ref{theorem1.1} in the range $s=2t\le k+1$ (see 
\cite[Theorem 2.1]{BW2017} in the case $r=1$ and $t=k$, and the earlier work of the author 
\cite[Lemma 2.1]{Woo2015} devoted to the case $k=3$ and $s=4$).\par

The conclusion of Theorem \ref{theorem1.1} is a consequence of a more general conclusion, as we now explain. We 
consider polynomials $\varphi_j\in\mathbb Z[t]$ $(1\le j\le k)$, and the Wronskian
\[
W_0(t;\boldsymbol \varphi)=\det \bigl( \varphi_j^{(i-1)}(t)\bigr)_{1\le i,j\le k},
\]
in which we write $\varphi_j^{(r)}(t)$ for the $r$-th derivative ${\rm d}^r\varphi_j(t)/{\rm d}t^r$. We introduce the 
exponential sum $F_k(\boldsymbol \alpha;X)=F_k(\boldsymbol \alpha;X;\boldsymbol \varphi)$, defined by
\begin{equation}\label{1.5}
F_k(\boldsymbol \alpha;X;\boldsymbol \varphi)=\sum_{1\le h\le X}\sum_{1\le z\le X}
e\bigl( h\varphi_1(z)\alpha_1 + \ldots + h\varphi_k(z)\alpha_k \bigr).
\end{equation}
Finally, when $s>0$, we define the mean value
\begin{equation}\label{1.6}
B_{s,k}(X;\boldsymbol \varphi)=\int_{[0,1)^k} \bigl| F_k(\boldsymbol \alpha;X;\boldsymbol \varphi) \bigr|^s\, 
{\rm d}\boldsymbol \alpha.
\end{equation}

\begin{theorem}\label{theorem1.2}
Suppose that $k\ge 1$ and $\varphi_j\in\mathbb{Z}[t]$ $(1\le j\le k)$ is a system of polynomials with 
$W_0(t;\boldsymbol \varphi)\ne 0$. Let $s$ be a positive number with $s\le k(k+1)/2$. Then for each $\varepsilon>0$, 
one has
\[
B_{s,k}(X;\boldsymbol \varphi) \ll X^{s+\varepsilon}.
\]
\end{theorem}

The conclusion of Theorem \ref{theorem1.2} shows that when $s\le k(k+1)/2$, the mean value 
$B_{s,k}(X;\boldsymbol \varphi)$ exhibits square-root cancellation relative to the trivial estimate 
\[
|F_k(\boldsymbol \alpha;X;\boldsymbol \varphi)|\le X^{2}.
\]
This new mean value estimate is notably sharp insofar as it applies to exponential sums not of translation-dilation 
invariant type.\par

Next, define
\begin{equation}\label{1.7}
g_k(\boldsymbol \alpha;X)=\sum_{1\le h\le X}\sum_{1\le z\le X}e\bigl( hz\alpha_1 + \ldots + h z^{k}\alpha_k \bigr),
\end{equation}
and, when $s>0$, put
\begin{equation}\label{1.8}
C_{s,k}(X)=\int_{[0,1)^k}\bigl| g_k(\boldsymbol \alpha;X) \bigr|^{s}\, {\rm d}\boldsymbol \alpha.
\end{equation}
Then, as a consequence of Theorem \ref{theorem1.2}, one has
\begin{equation}\label{1.9}
C_{s,k}(X) \ll X^{s+\varepsilon}
\end{equation}
whenever $s\le\tfrac{1}{2}k(k+1)$. By applying ideas aligned with our previous work \cite{Woo2026a} relating to 
relatives of Vinogradov's system, we are able to obtain sharp estimates for all $s>0$.

\begin{theorem}\label{theorem1.3}
Suppose that $k\ge 1$ and $s>0$. Then for each $\varepsilon>0$, one has
\[
C_{s,k}(X)\ll X^\varepsilon\bigl( X^s+X^{2s-k(k+3)/2}\bigr) .
\]
\end{theorem}

Thus, one sees that the bound \eqref{1.9} holds whenever $s\le\tfrac{1}{2}k(k+3)$. Note that since the sum of the 
degrees of the monomials $hz,\ldots,hz^{k}$ is equal to
\[
\sum_{j=2}^{k+1}j= \tfrac{1}{2}k(k+3),
\]
the upper bound for $C_{s,k}(X)$ recorded in Theorem \ref{theorem1.3} achieves that predicted by the standard 
heuristics of the subject. Indeed, the lower bound
\[
C_{s,k}(X)\gg X^{s}+ X^{s-k(k+3)/2}
\]
follows via a consideration of diagonal solutions and the contribution from a single major arc centred at $\mathbf 0$.

\par The proofs of Theorems \ref{theorem1.1}, \ref{theorem1.2} and \ref{theorem1.3} all depend, in one way or 
another, on estimates for mean values of a weighted analogue of the exponential sum 
$F_k(\boldsymbol \alpha;X;\boldsymbol \varphi)$. Here, we adapt the methods introduced in 
\cite[\S\S2 and 3]{Woo2026a} to the present setting. Thus, on defining
\begin{equation}\label{1.10}
G_k(\boldsymbol \alpha;X;\boldsymbol \varphi)=\sum_{1\le z\le X}\biggl| \sum_{1\le h\le X}
e\bigl( h\varphi_1(z)\alpha_1 + \ldots + h\varphi_k(z)\alpha_k \bigr)\biggr|^2,
\end{equation}
we find that the mean value
\begin{equation}\label{1.11}
A^*_{s,k}(X;\boldsymbol \varphi)=\int_{[0,1)^k} \bigl| G_k(\boldsymbol \alpha;X;\boldsymbol \varphi) \bigr|^s
\, {\rm d}\boldsymbol \alpha
\end{equation}
may be related to the mean value $B_{s,k}(X;\boldsymbol \varphi)$ defined via \eqref{1.6}. When $s$ is even, this is 
more or less immediate via orthogonality, but in general this relationship is far more subtle to establish, since
Diophantine interpretations are then absent. It transpires that the mean value $A^*_{s,k}(X;\boldsymbol \varphi)$ is 
readily amenable to estimates available from our earlier work \cite{Woo2026} based on efficient congruencing in the 
ring $\mathbb Z[t]$ modulo $t^2$.\par

We establish the relationship between our work \cite{Woo2026} and the mean value 
$A^*_{s,k}(X;\boldsymbol \varphi)$ in \S2, thus laying the groundwork for what is to come. Then, in \S3, we show 
how the mean value $A^*_{s,k}(X;\boldsymbol \varphi)$ may be related to the corresponding unweighted mean 
value $B_{s,k}(X;\boldsymbol \varphi)$. This work establishes Theorem \ref{theorem1.2}. Little additional work is 
required to extract the conclusion of Theorem \ref{theorem1.1}, and this we also complete in \S3. The proof of 
Theorem \ref{theorem1.3} is discussed in \S4 via the application of a kind of translation-dilation invariance idea 
previously utilised in earlier work of the author (see \cite{Woo2012, Woo2026a}).\par

Throughout, we use $X$ to denote an integer sufficiently large in terms of all ambient parameters, such as 
$\varepsilon$, $k$ and $s$. We make use of standard notational conventions applied in this area of mathematics. 
Thus, throughout, we make liberal use of vector notation in settings not always conventional in nature. For example, 
we may write $1\le\mathbf x\le X$ to denote that every coordinate $x_i$ of $\mathbf x$ satisfies $1\le x_i\le X$. Also, 
given a situation in which the parameter $\varepsilon$ has not already been fixed, in any statement involving the 
letter $\varepsilon$, it is implicitly asserted that the statement holds for each $\varepsilon>0$. In such circumstances, 
implicit constants in Vinogradov's notation $\ll$ and $\gg$ may depend on $\varepsilon$.\par

Work on this paper was conducted while the author was supported by NSF grants DMS-2001549 and DMS-2502625, 
and by Simons Fellowship in Mathematics SFM-00011955. The author is grateful to the Institute for Advanced Study, 
Princeton, for hosting his sabbatical, during which period this paper was completed.

\section{An auxiliary mean value theorem}
We make use of estimates associated with a fully translation-dilation invariant system so as to derive upper bounds 
for the mean value $A^{*}_{s,k}(X;\boldsymbol \varphi)$. We achieve a relationship between these quantities by 
exploiting an idea introduced in \cite[\S2]{Woo2026a}.

\begin{lemma}\label{lemma2.1}
Suppose that $k\ge 1$ and $\varphi_j\in\mathbb{Z}[t]$ $(1\le j\le k)$ is a system of polynomials with 
$W_0(t;\boldsymbol \varphi)\ne 0$. Let $s$ be a positive number with $s\le k(k+1)/2$. Then for each 
$\varepsilon>0$, one has
\[
A^{*}_{s,k}(X;\boldsymbol \varphi) \ll X^{2s+\varepsilon}.
\]
\end{lemma}

\begin{proof} We relate the mean value $A^{*}_{s,k}(X;\boldsymbol \varphi)$ to one that envelopes the underlying 
Diophantine system, in a suitable sense. First, for $1\le j\le k$, we define the polynomial $\psi_j(t)$ by means of the
relation $\psi_j'(t)=\varphi_j(t)$, subject to the condition $\psi_j(0)=0$. Thus, the polynomial $\psi_j$ is an 
antiderivative of $\varphi_j$. By change of variable in the mean value $A^{*}_{s,k}(X;\boldsymbol \varphi)$, there is no 
loss of generality if we replace $\varphi_j(t)$ by $\lambda\varphi_j(t)$, for any non-zero integer $\lambda$. By making 
a choice for $\lambda$ having the property that $\lambda\psi_j(t)\in\mathbb{Z}[t]$ $(1\le j\le k)$, we see that there is 
likewise no loss of generality in supposing, {\it ab initio}, that $\varphi_j(t)\in\mathbb{Z}[t]$ $(1\le j\le k)$ is a system 
of polynomials having the property that $\psi_j(t)\in\mathbb{Z}[t]$ $(1\le j\le k)$. The purpose of this manoeuvre is to 
ensure that
\[
\psi_j(u+vT) \equiv \psi_j(u)+v\varphi_j(u)\,T \mmod{T^2},
\]
thereby opening the possibility of making use of \cite[Theorem 1.1 and Corollary 1.2]{Woo2026}.\par

Next, when $\boldsymbol \alpha,\boldsymbol \beta\in\mathbb{R}^k$, we define the polynomials
\[
\Psi_k(z;\boldsymbol \alpha) = \alpha_1\psi_1(z) + \ldots + \alpha_k\psi_k(z),
\]
and
\begin{equation}\label{2.1}
\Phi_k(z;\boldsymbol \beta) = \beta_1\varphi_1(z) + \ldots + \beta_k\varphi_k(z).
\end{equation}
We then define the exponential sum
\[
b_k(\boldsymbol \alpha,\boldsymbol \beta;X)=\sum_{1\le z\le X}\sum_{1\le h\le X}
e\bigl( \Psi_k(z;\boldsymbol \alpha)+h\Phi_k(z;\boldsymbol \beta) \bigr),
\]
and the associated mean value
\begin{equation}\label{2.2}
B_{s,k}(X)=\int_{[0,1)^k}\int_{[0,1)^k}\bigl| b_k(\boldsymbol \alpha,\boldsymbol \beta;X) \bigr|^{2s}\,{\rm d}
\boldsymbol \alpha\,{\rm d}\boldsymbol \beta.
\end{equation}

\par Given an integral $k$-tuple $\mathbf u$, denote by $\mathcal{B}(\mathbf u)$ the set of integers $z$ with 
$1\le z\le X$ having the property that
\[
\psi_j(z) = u_j \quad (1\le j\le k).
\]
Note that $\text{card}(\mathcal{B}(\mathbf u))$ is bounded above by the maximum of the degrees of $\psi_j$ for 
$1\le j\le k$. Equipped with this notation, it follows via orthogonality that
\begin{equation}\label{2.3}
\int_{[0,1)^k} \bigl| b_k(\boldsymbol \alpha,\boldsymbol \beta;X) \bigr|^2\,{\rm d}\boldsymbol \alpha
=\sum_{\mathbf u\in\mathbb{Z}^k}\Biggl| \sum_{\substack{1\le z\le X \\ z\in\mathcal{B}(\mathbf u)}}
\sum_{1\le h\le X}e\bigl( h\Phi_k(z;\boldsymbol \beta) \bigr)\Biggr|^2.
\end{equation}

\par We may suppose that $X$ is sufficiently large in terms of the coefficients of $\boldsymbol \varphi$, and hence 
also $\boldsymbol \psi$, and consequently the real zeros of $\varphi_j(t)$ may be considered bounded in terms of
$\boldsymbol \varphi$. In particular, we have $|\varphi_j(\xi)|\gg 1$ whenever $\xi$ is sufficiently large. Hence, by the 
mean value theorem, whenever $z_1$ and $z_2$ are sufficiently large with $z_2>z_1$, there exists $\xi\in(z_1,z_2)$ 
with
\[
\biggl| \frac{\psi_j(z_1)-\psi_j(z_2)}{z_1-z_2}\biggr|=|\psi_j'(\xi)|=|\varphi_j(\xi)|\gg 1.
\]
Consequently, we find that $\text{card}(\mathcal{B}(\mathbf u))=1$ whenever $u_j$ is sufficiently large for any index 
$j$ with $1\le j\le k$. By making use of this observation within \eqref{2.3}, we find that for some positive number
$c=c(\boldsymbol \varphi)$, one has
\begin{align*}
\int_{[0,1)^k} \bigl| b_k(\boldsymbol \alpha,\boldsymbol \beta;X) \bigr|^{2}\,{\rm d}\boldsymbol \alpha
&= \sum_{\substack{\mathbf u\in\mathbb{Z}^{k} \\ \text{card}(\mathcal{B}(\mathbf u))=1}}\,
\sum_{\substack{1\le z\le X\\ z\in\mathcal B(\mathbf u)}}\,\biggl| \sum_{1\le h\le X}
e\bigl( h\Phi_k(z;\boldsymbol \beta)\bigr)\biggr|^2\\
&\qquad \qquad + O\Biggl( \sum_{\substack{\mathbf u\in\mathbb{Z}^{k} \\ |\mathbf u|\le c(\boldsymbol \varphi)}}
\Biggl| \sum_{\substack{1\le z\le X \\ z\in\mathcal{B}(\mathbf u)}}\, \sum_{1\le h\le X} 
e\bigl( h\Phi_k(z;\boldsymbol \beta) \bigr)\Biggr|^2\Biggr) \\
&= \sum_{1\le z\le X} \biggl| \sum_{1\le h\le X} e\bigl( h\Phi_k(z;\boldsymbol \beta) \bigr) \biggr|^2+ O(X^2).
\end{align*}
Hence, we conclude from \eqref{1.10} that
\begin{equation}\label{2.4}
\int_{[0,1)^k}\bigl| b_k(\boldsymbol \alpha,\boldsymbol \beta;X)\bigr|^2\, {\rm d}\boldsymbol \alpha
=G_k(\boldsymbol \beta;X;\boldsymbol \varphi) + O(X^2).
\end{equation}

By applying H\"older's inequality to the relation \eqref{2.4} and recalling \eqref{2.2}, we deduce from \eqref{1.11} that
\begin{align}
A^*_{s,k}(X;\boldsymbol \varphi)&\ll \int_{[0,1)^k}\biggl( \int_{[0,1)^k}
\bigl| b_k(\boldsymbol \alpha,\boldsymbol \beta;X) \bigr|^2\,{\rm d}\boldsymbol \alpha\biggr)^{\!\!s}\,{\rm d}
\boldsymbol \beta + O(X^{2s}) \notag \\
&\ll \int_{[0,1)^k}\int_{[0,1)^k}\bigl| b_k(\boldsymbol \alpha,\boldsymbol \beta;X) \bigr|^{2s}\,{\rm d}
\boldsymbol \alpha\,{\rm d}\boldsymbol \beta + O(X^{2s})\notag \\
&=B_{s,k}(X)+O(X^{2s}).\label{2.5}
\end{align}
The mean value $B_{s,k}(X)$ is the subject of \cite[Corollary 1.2]{Woo2026} in the situation that $s$ is an integer with 
$s\le k(k+1)/2$. Indeed, we find via orthogonality that $B_{s,k}(X)$ counts the number of integral solutions of the 
Diophantine system
\[
\left. \begin{aligned}
\sum_{i=1}^s\bigl( \psi_j(u_i) - \psi_j(u_{s+i}) \bigr) &= 0, \\
\sum_{i=1}^s\bigl( v_i\psi_j'(u_i) - v_{s+i}\psi_j'(u_{s+i}) \bigr) &= 0
\end{aligned} \right\} \quad (1\le j\le k),
\]
with $1\le \mathbf u, \mathbf v\le X$. Moreover, since
\[
\det\bigl( \psi_j^{(i)}(t) \bigr)_{1\le i,j\le k}=\det\bigl( \varphi_j^{(i-1)}(t)\bigr)_{1\le i,j\le k}=W_0(t;\boldsymbol \varphi),
\]
one sees that the Wronskian condition of \cite[Corollary 1.2]{Woo2026} is satisfied whenever one has 
$W_0(t;\boldsymbol \varphi)\ne 0$, as we suppose to be the case. On writing $s_0=k(k+1)/2$, we thus conclude from 
\cite[Corollary 1.2]{Woo2026} that $B_{s_0,k}(X)\ll X^{2s_0+\varepsilon}$. An application of H\"older's inequality 
therefore shows from \eqref{2.2} that whenever $0<s\le s_0$, one has
\[
B_{s,k}(X)\ll \bigl( B_{s_0,k}(X)\bigr)^{s/s_0}\ll X^{2s+\varepsilon}.
\]
Thus, we deduce from \eqref{2.5} that $A^{*}_{s,k}(X;\boldsymbol \varphi) \ll X^{2s+\varepsilon}$, and the proof of 
the lemma is complete.
\end{proof}

\section{Sharp moment estimates} In this section we seek to establish Theorems \ref{theorem1.1} and 
\ref{theorem1.2}. It is worth noting that the conclusions of these theorems are immediate corollaries of Lemma 
\ref{lemma2.1} when $s$ is even and $s\le\tfrac{1}{2}k(k+1)$, as we now explain. Indeed, in the latter situation, we 
may write $s=2t$ for an integer $t$, and then it follows via orthogonality that $A^*_{2t,k}(2X;\boldsymbol \varphi)$ 
counts the number of integral solutions of the Diophantine system
\[
\sum_{i=1}^t (g_i-g_i')\varphi_j(z_i)=\sum_{i=1}^t (g_{t+i}-g_{t+i}')\varphi_j(z_{t+i})\quad (1\le j\le k),
\]
with $1\le \mathbf g, \mathbf g', \mathbf z\le 2X$. For a given integer $h$ with $|h|\le 2X$, the number of solutions of the 
equation $g-g'=h$, with $1\le g,g'\le 2X$, is precisely $2X-|h|$. Meanwhile, from \eqref{1.5} and \eqref{1.6}, again by 
orthogonality, we see that $B_{2t,k}(X;\boldsymbol \varphi)$ counts the number of integral solutions of the Diophantine 
system
\[
\sum_{i=1}^{t} h_i\varphi_j(z_i)= \sum_{i=1}^{t} h_{t+i}\varphi_j(z_{t+i})\quad (1\le j\le k),
\]
with $1\le \mathbf h, \mathbf z\le X$. Since $2X-|h|\ge X$ when $1\le h\le X$, it follows that
\[
A^*_{2t,k}(2X;\boldsymbol \varphi) \ge X^{2t}B_{2t,k}(X;\boldsymbol \varphi).
\]
Consequently, we find from Lemma \ref{lemma2.1} that when $2t\le k(k+1)/2$, one has
\[
B_{2t,k}(X;\boldsymbol \varphi) \le X^{-2t}A^*_{2t,k}(2X;\boldsymbol \varphi) \ll X^{2t+\varepsilon}.
\]
This establishes Theorem \ref{theorem1.2} for $s$ even and $s\le\tfrac{1}{2}k(k+1)$, with Theorem \ref{theorem1.1} 
following as a corollary under the same hypotheses on $s$.\par

The situation with $s$ odd is not amenable to a Diophantine interpretation of the mean value 
$B_{s,k}(X;\boldsymbol \varphi)$, and so we must work rather harder to obtain the same conclusion. We begin by 
establishing a technical lemma. The announcement of this lemma requires that we introduce some additional notation. 
First, we recall \eqref{2.1} and define
\[
\Gamma_k(\boldsymbol \alpha;\beta,\gamma;X)=\sum_{1\le z\le X}\biggl| \sum_{1\le h\le X}
e\bigl( h\Phi_k(z;\boldsymbol \alpha)+\beta h\bigr)\biggr|^2 e(\gamma z).
\]
The reader may care to note, by reference to \eqref{1.10}, that one has
\[
\Gamma_k(\boldsymbol \alpha;0,0;X)=G_k(\boldsymbol \alpha;X;\boldsymbol \varphi).
\]
We also require the auxiliary exponential sum
\begin{equation}\label{3.1}
K(\theta;X)=\sum_{1\le y\le X} e(-\theta y).
\end{equation}

\begin{lemma}\label{lemma3.1}
One has
\[
XF_k(\boldsymbol \alpha;X;\boldsymbol \varphi)=\int_0^1\!\!\int_0^1 \bigl( \Gamma_k(\boldsymbol \alpha;\beta,\gamma;2X)- 
\Gamma_k(\boldsymbol \alpha;\beta,\gamma;X) \bigr) K(\beta;X)K(\gamma;X)\,{\rm d}\beta\,{\rm d}\gamma.
\]
\end{lemma}

\begin{proof}
Observe that the number of solutions of the equation $h_1-h_2=h$ with $1\le h_i\le X$ $(i=1,2)$ is $X-|h|$, when 
$|h|\le X$, and otherwise is $0$. Thus, we find that
\[
\Gamma_k(\boldsymbol \alpha;\beta,\gamma;X)=\sum_{1\le z\le X}\sum_{|h|\le X}(X-|h|)
e\bigl( h\Phi_k(z;\boldsymbol \alpha)+\beta h+\gamma z\bigr).
\]
It follows that
\begin{equation}\label{3.2}
\Gamma_k(\boldsymbol \alpha;\beta,\gamma;2X)-\Gamma_k(\boldsymbol \alpha;\beta,\gamma;X)=
\sum_{i=1}^3 T_i(\boldsymbol \alpha;\beta,\gamma;X),
\end{equation}
where
\begin{align*}
T_1(\boldsymbol \alpha;\beta,\gamma;X)&=X\sum_{1\le z\le X}\sum_{|h|\le X}
e\bigl( h\Phi_k(z;\boldsymbol \alpha)+\beta h+\gamma z \bigr),\\
T_2(\boldsymbol \alpha;\beta,\gamma;X)&=\sum_{X<z\le 2X}\sum_{|h|\le 2X}(2X-|h|)
e\bigl( h\Phi_k(z;\boldsymbol \alpha)+ \beta h + \gamma z \bigr),\\
T_3(\boldsymbol \alpha;\beta,\gamma;X)&=\sum_{1\le z\le X}\sum_{X<|h|\le 2X}(2X-|h|)
e\bigl( h\Phi_k(z;\boldsymbol \alpha)+\beta h+\gamma z\bigr).
\end{align*}

\par On recalling \eqref{1.5}, we find that as a consequence of orthogonality, one has the relation
\begin{align}
\int_0^1\!\!\int_0^1 &T_1(\boldsymbol \alpha;\beta,\gamma;X)K(\beta;X)K(\gamma;X)\,{\rm d}\beta\,{\rm d}\gamma 
\notag \\
&= X\int_0^1\!\!\int_0^1 \sum_{1\le z\le X}\sum_{|h|\le X}
e\bigl( h\Phi_k(z;\boldsymbol \alpha)+\beta h+\gamma z\bigr) \sum_{1\le y\le X}e(-\beta y)
\sum_{1\le w\le X}e(-\gamma w)\,{\rm d}\beta\,{\rm d}\gamma \notag \\
&=XF_k(\boldsymbol \alpha;X;\boldsymbol \varphi ).\label{3.3}
\end{align}
Meanwhile, since the summands in $T_2(\boldsymbol \alpha;\beta,\gamma;X)$ satisfy $z>X$, we see in a similar 
manner that
\begin{align*}
\int_0^1 T_2&(\boldsymbol \alpha;\beta,\gamma;X)K(\gamma;X)\,{\rm d}\gamma \\
&=\sum_{X<z\le 2X}\sum_{|h|\le 2X}(2X-|h|)e\bigl( h\Phi_k(z;\boldsymbol \alpha)+\beta h\bigr)
\int_0^1 e(\gamma z)\sum_{1\le w\le X}e(-\gamma w)\,{\rm d}\gamma = 0.
\end{align*}
Consequently, we have
\begin{equation}\label{3.4}
\int_0^1\!\!\int_0^1 T_2(\boldsymbol \alpha;\beta,\gamma;X)K(\beta;X)K(\gamma;X)\,{\rm d}\beta\,{\rm d}\gamma=0.
\end{equation}
A similar argument shows that
\begin{equation}\label{3.5}
\int_0^1\!\!\int_0^1 T_3(\boldsymbol \alpha;\beta,\gamma;X)K(\beta;X)K(\gamma;X)\,{\rm d}\beta\,{\rm d}\gamma=0.
\end{equation}
The desired conclusion now follows on combining \eqref{3.2}-\eqref{3.5}.
\end{proof}

The conclusion of Lemma \ref{lemma3.1} permits the estimation of mean values of the exponential sum 
$F_k(\boldsymbol \alpha;X;\boldsymbol \varphi )$ in terms of $A^*_{s,k}(X;\boldsymbol \varphi)$.

\begin{lemma}\label{lemma3.2}
One has
\[
B_{s,k}(X;\boldsymbol \varphi)\ll X^{-s} (\log(2X))^{2s}\bigl( A^{*}_{s,k}(X;\boldsymbol \varphi) 
+A^*_{s,k}(2X;\boldsymbol \varphi) \bigr).
\]
\end{lemma}

\begin{proof}
We find from Lemma \ref{lemma3.1} and \eqref{1.6} that $B_{s,k}(X;\boldsymbol \varphi)$ is equal to
\[
X^{-s} \int_{[0,1)^k}\biggl| \int_0^1\!\!\int_0^1 \bigl( \Gamma_k(\boldsymbol \alpha;\beta,\gamma;2X) - 
\Gamma_k(\boldsymbol \alpha;\beta,\gamma;X) \bigr) K(\beta;X)K(\gamma;X)\,{\rm d}\beta\,{\rm d}\gamma \biggr|^s
\,{\rm d}\boldsymbol \alpha.
\]
By H\"older's inequality, one finds that
\begin{multline*}
\biggl| \int_0^1\!\!\int_0^1 \bigl( \Gamma_k(\boldsymbol \alpha;\beta,\gamma;2X)-
\Gamma_k(\boldsymbol \alpha;\beta,\gamma;X) \bigr) K(\beta;X)\,K(\gamma;X)\,{\rm d}\beta\,{\rm d}\gamma 
\biggr|^s\\
\ll \biggl( \int_0^1\!\!\int_0^1 |K(\beta;X)\,K(\gamma;X)|\,{\rm d}\beta\,{\rm d}\gamma \biggr)^{s-1}
\bigl( I(\boldsymbol \alpha;2X) + I(\boldsymbol \alpha;X) \bigr),
\end{multline*}
where
\begin{equation}\label{3.6}
I(\boldsymbol \alpha;Y)=\int_0^1\!\!\int_0^1 \bigl| \Gamma_k(\boldsymbol \alpha;\beta,\gamma;Y) \bigr|^s
|K(\beta;X)\,K(\gamma;X)|\,{\rm d}\beta\,{\rm d}\gamma.
\end{equation}
Since
\begin{equation}\label{3.7}
\int_0^1 |K(\theta;X)|\,{\rm d}\theta \ll \int_0^1 \min\{ X, \lVert\theta\rVert^{-1} \}\,{\rm d}\theta \ll \log(2X),
\end{equation}
we discern thus far that
\begin{equation}\label{3.8}
B_{s,k}(X;\boldsymbol \varphi)\ll X^{-s} (\log(2X))^{2s-2}\int_{[0,1)^k}\bigl( I(\boldsymbol \alpha;2X)+ 
I(\boldsymbol \alpha;X)\bigr)\,{\rm d}\boldsymbol \alpha.
\end{equation}

\par The definition of $\Gamma_k$ combines with \eqref{3.6} and \eqref{3.7} to show that
\begin{align*}
I(\boldsymbol \alpha;Y)&\le \int_0^1\!\!\int_0^1\biggl( \sum_{1\le z\le Y}\Bigl| \sum_{1\le h\le Y}e\bigl( h\Phi_k(z;\boldsymbol \alpha)
+\beta h\bigr) \Bigr|^2 \biggr)^{\! s}|K(\beta;X)\,K(\gamma;X)|\,{\rm d}\beta\,{\rm d}\gamma \\
&\ll (\log 2X) \int_0^1\biggl( \sum_{1\le z\le Y}\Bigl| \sum_{1\le h\le Y}e\bigl( h\Phi_k(z;\boldsymbol \alpha)
+\beta h \bigr)\Bigr|^2 \biggr)^{\! s}|K(\beta;X)|\,{\rm d}\beta.
\end{align*}
Hence, we infer that
\begin{equation}\label{3.9}
\int_{[0,1)^k}I(\boldsymbol \alpha;Y)\,{\rm d}\boldsymbol \alpha \ll (\log (2X))\int_0^1 T(\beta;Y)|K(\beta;X)|\,{\rm d}
\beta,
\end{equation}
where we write
\[
T(\beta;Y)=\int_{[0,1)^k}\biggl( \sum_{1\le z\le Y}\Bigl| \sum_{1\le h\le Y}e\bigl( h\Phi_k(z;\boldsymbol \alpha)
+\beta h\bigr) \Bigr|^2\biggr)^{\! s}\,{\rm d}\boldsymbol \alpha.
\]

\par By orthogonality, the integral $T(\beta;Y)$ counts the integral solutions of the system of equations
\[
\sum_{i=1}^s (h_i-h_i')\varphi_j(z_i)=0\quad (1\le j\le k),
\]
with $1\le \mathbf z, \mathbf h, \mathbf h'\le Y$, and with each solution $\mathbf z, \mathbf h, \mathbf h'$ counted 
with weight
\[
e\biggl( \beta \sum_{i=1}^s (h_i-h_i') \biggr).
\]
Thus, on recalling \eqref{1.11}, we discern that
\[
|T(\beta;Y)|\le T(0;Y)=A^*_{s,k}(Y;\boldsymbol \varphi).
\]
On substituting this upper bound into \eqref{3.9}, we conclude via \eqref{3.7} and \eqref{3.8} that
\begin{align*}
B_{s,k}(X;\boldsymbol \varphi)&\ll X^{-s}(\log (2X))^{2s-1}\bigl( A^*_{s,k}(X;\boldsymbol \varphi)+
A^*_{s,k}(2X;\boldsymbol \varphi)\bigr) \int_0^1\min\{2X,\|\theta\|^{-1}\}\,{\rm d}\theta \\
&\ll X^{-s}(\log (2X))^{2s}\bigl( A^*_{s,k}(X;\boldsymbol \varphi)+A^*_{s,k}(2X;\boldsymbol \varphi)\bigr) .
\end{align*}
This completes the proof of the lemma.
\end{proof}

We are now equipped to complete the proof of Theorem \ref{theorem1.2}

\begin{proof}[The proof of Theorem \ref{theorem1.2}.]
On substituting the upper bound provided by Lemma \ref{lemma2.1} into the conclusion of Lemma \ref{lemma3.2}, 
we find that
\[
B_{s,k}(X;\boldsymbol \varphi)\ll X^{-s}(\log (2X))^{2s}\cdot (2X)^{2s+\varepsilon}\ll X^{s+2\varepsilon}.
\]
This completes the proof of Theorem \ref{theorem1.2}.
\end{proof}

Finally, the conclusion of Theorem \ref{theorem1.1} follows as an immediate corollary of Theorem \ref{theorem1.2} on 
taking $\varphi_j(z)=z^{j-1}$ $(1\le j\le k)$, since then $W_0(t;\boldsymbol \varphi)\ne 0$.

\section{Sharp moment estimates with a missing slice}
We move in this section towards the goal of establishing Theorem \ref{theorem1.3}. This conclusion is a mean value 
estimate related to the upper bound provided in Theorem \ref{theorem1.1} for the mean value of exponential sums 
$f_k(\boldsymbol \alpha;X)$, save that the term $h\alpha_1$ in the argument of the exponential sum is removed in 
\eqref{1.1}. In a related setting (see \cite[\S3]{Woo2026a}), such a situation was handled by a shift towards a 
translation-dilation invariant system of equations. We adapt that strategy now to the present setting.\par

Before announcing our key lemma, we begin by recalling the exponential sums $g_k(\boldsymbol \alpha;X)$ and 
$f_k(\boldsymbol \alpha;X)$ defined in \eqref{1.7} and \eqref{1.1}, respectively. We again introduce an auxiliary sum as in 
the preamble to Lemma \ref{lemma3.1}. We define
\[
\Lambda_k(\boldsymbol \alpha;\beta,\gamma;X)=\sum_{1\le z\le X}\biggl| \sum_{1\le h\le X}e\bigl( hz\alpha_1+\ldots 
+hz^k\alpha_k + \beta h\bigr)\biggr|^2 e(\gamma z),
\]
and we recall the definition of $K(\theta;X)$ from \eqref{3.1}.

\begin{lemma}\label{4.1}
One has
\[
Xg_k(\boldsymbol \alpha;X)=\int_0^1\!\!\int_0^1 \bigl( \Lambda_k(\boldsymbol \alpha;\beta,\gamma;2X) - 
\Lambda_k(\boldsymbol \alpha;\beta,\gamma;X) \bigr) K(\beta;X)\,K(\gamma;X)\,{\rm d}\beta\,{\rm d}\gamma.
\]
\end{lemma}

\begin{proof}
This is the special case of Lemma \ref{lemma3.1} in which $\varphi_j(z)=z^j$ $(1\le j\le k)$.
\end{proof}

We are now equipped to bound the mean value $C_{s,k}(X)$, defined in \eqref{1.8}, in terms of mean values involving 
auxiliary exponential sums. We introduce the exponential sum
\[
h_k(\boldsymbol \alpha;X)=\sum_{1\le z\le X}\biggl| \sum_{1\le h\le X}e\bigl( hz\alpha_1+\ldots +h z^k\alpha_k \bigr)
\biggr|^2,
\]
and then define a mean value related to $C_{s,k}(X)$, namely
\[
C^*_{s,k}(X)=\int_{[0,1)^k}\bigl| h_k(\boldsymbol \alpha;X)\bigr|^s\, {\rm d}\boldsymbol \alpha.
\]

\begin{lemma}\label{lemma4.2}
When $s$ is a natural number, one has
\[
C_{s,k}(X)\ll X^{-s}(\log (2X))^{2s}\bigl( C^*_{s,k}(X) + C^*_{s,k}(2X) \bigr).
\]
\end{lemma}

\begin{proof}
This is the special case of Lemma \ref{lemma3.2} corresponding to the choice of polynomials $\varphi_j(z)=z^j$ 
$(1\le j\le k)$.
\end{proof}

We can now begin the main discussion.

\begin{proof}[The proof of Theorem \ref{theorem1.3}.]
We follow an argument motivated by that of Lemma \ref{lemma3.1}, fixing $s=\tfrac{1}{2}(k^{2}+3k)$, and focusing on the 
estimation of $C^*_{s,k}(Y)$. Here, we may take $Y=X$ or $Y=2X$. Define $\delta_j$ to be $1$ when $j=0$, and $0$ 
otherwise. Then, by interpreting the mean value $C^*_{s,k}(Y)$ via orthogonality, one finds that it counts the integral 
solutions of the system of equations
\begin{equation}\label{4.1}
\sum_{i=1}^{s} (h_i - g_i)z_i^j = \delta_jh\quad (0\le j\le k),
\end{equation}
with $1\le \mathbf h, \mathbf g, \mathbf z\le Y$ and $|h|\le sY$. Here, we adopt the usual convention that $z_i^{0}=1$. 
Denote by $I^*_{s,k}(Y)$ the number of solutions of the system of equations
\begin{equation}\label{4.2}
\sum_{i=1}^s(h_i-g_i)z_i^j=\delta_j(l_1-l_2)\quad (0\le j\le k),
\end{equation}
with $1\le \mathbf h, \mathbf g, \mathbf z\le Y$ and $|\mathbf l|\le 2sY$. Then, a comparison of \eqref{4.1} and 
\eqref{4.2} reveals that
\begin{equation}\label{4.3}
C^*_{s,k}(Y)\le Y^{-1}I^*_{s,k}(Y).
\end{equation}

\par We next consider a shift of each variable $z_i$ in \eqref{4.2} by any integer $w$ with $1\le w\le Y$. By the 
binomial theorem, for each shift $w$, one finds that $\mathbf g,\mathbf h,\mathbf z$ is a solution of \eqref{4.2} if and 
only if it is also a solution of the system
\[
\sum_{i=1}^s(h_i-g_i)(z_i+w)^j=(l_1-l_2)w^j\quad (0\le j\le k).
\]
We deduce that for each fixed integer $w$ with $1\le w\le Y$, the mean value $I^*_{s,k}(Y)$ is bounded above by the 
number of solutions of the system
\begin{equation}\label{4.4}
\sum_{i=1}^s(h_i-g_i)u_i^j=(l_1-l_2)w^j\quad (0\le j\le k),
\end{equation}
with $1\le \mathbf g,\mathbf h,\mathbf u\le 2Y$ and $|\mathbf l|\le 2sY$.\par

Define the exponential sums
\begin{equation}\label{4.5}
r_k(\boldsymbol \alpha;Y)=\sum_{1\le z\le Y}\biggl| \sum_{1\le h\le Y}e\bigl( h\alpha_0+h z\alpha_1+\ldots 
+h z^k\alpha_k \bigr)\biggr|^2,
\end{equation}
and
\[
r^*_{k,s}(\boldsymbol \alpha;w;Y)=\sum_{1\le h\le 2sY}e\bigl( h\alpha_0+hw\alpha_1+\ldots +hw^k\alpha_k \bigr).
\]
Then, by averaging over the above choices for $w$ and applying orthogonality, we deduce by means of \eqref{4.4} 
that
\begin{align*}
I^*_{s,k}(Y)&\le Y^{-1} \sum_{1\le w\le Y}\int_{[0,1)^{k+1}}\bigl| r_k(\boldsymbol \alpha;2Y) \bigr|^s
\bigl| r^*_{k,s}(\boldsymbol \alpha;w;Y)\bigr|^2\,{\rm d}\boldsymbol \alpha \\
&\le Y^{-1}\sum_{1\le w\le 2sY}\int_{[0,1)^{k+1}}\bigl| r_k(\boldsymbol \alpha;2Y) \bigr|^s 
\bigl| r^*_{k,s}(\boldsymbol \alpha;w;Y) \bigr|^2\,{\rm d}\boldsymbol \alpha \\
&=Y^{-1}\int_{[0,1)^{k+1}}\bigl| r_k(\boldsymbol \alpha;2Y) \bigr|^s\bigl| r_k(\boldsymbol \alpha;2sY) \bigr|
\,{\rm d}\boldsymbol \alpha.
\end{align*}
Thus, on applying \eqref{4.3} and H\"older's inequality, we deduce that
\begin{equation}\label{4.6}
C^*_{s,k}(Y)\ll Y^{-2}J_2^{s/(s+1)}J_{2s}^{1/(s+1)},
\end{equation}
where
\[
J_t=\int_{[0,1)^{k+1}}\bigl| r_k(\boldsymbol \alpha;tY) \bigr|^{s+1}\,{\rm d}\boldsymbol \alpha.
\]

\par Recall the definition \eqref{1.11} and put $\varphi_j(z)=z^{j-1}$ $(1\le j\le k+1)$. Then, a comparison of 
\eqref{4.5} with \eqref{1.10} shows that
\[
J_t=A^*_{s+1,k+1}(tY).
\]
Note next that since $s=\tfrac{1}{2}k(k+3)$, one has
\[
s+1=\tfrac{1}{2}(k^2+3k+2)=\tfrac{1}{2}(k+1)(k+2).
\]
Hence, by Lemma \ref{lemma2.1}, one finds from \eqref{4.6} that
\[
C^*_{s,k}(Y)\ll Y^{-2}\cdot Y^{2s+2+\varepsilon}=Y^{2s+\varepsilon}.
\]
On substituting this estimate into Lemma \ref{lemma4.2}, we conclude that
\[
C_{s,k}(X)\ll X^{-s}\cdot(\log (2X))^{2s}\cdot X^{2s+\varepsilon}\ll X^{s+2\varepsilon}.
\]
This estimate has been shown to hold for $s=\tfrac{1}{2}k(k+3)$. The proof of Theorem \ref{theorem1.3} follows by 
applying H\"older's inequality to discern that for $0\le u\le \tfrac{1}{2}k(k+3)$, one has
\[
C_{u,k}(X)\le (C_{s,k}(X))^{u/s}\ll X^{u+\varepsilon}, 
\]
and applying a trivial estimate for $|g_k(\boldsymbol \alpha;X)|$ to confirm that for $u>\tfrac{1}{2}k(k+3)$, one has
\[
C_{u,k}(X)\ll (X^2)^{u-s}C_{s,k}(X)\ll X^{2u-\tfrac{1}{2}k(k+3)+\varepsilon}.
\]
This completes the proof of the theorem.
\end{proof}

\bibliographystyle{amsbracket}
\providecommand{\bysame}{\leavevmode\hbox to3em{\hrulefill}\thinspace}

\end{document}